\RequirePackage{tikz}
\documentclass[sn-mathphys]{sn-jnl}

\usepackage{natbib}

\theoremstyle{thmstyleone}%
\newtheorem{theorem}{Theorem}

\newtheorem{corollary}[theorem]{Corollary}
\newtheorem{proposition}[theorem]{Proposition}%

\theoremstyle{thmstyletwo}%
\newtheorem{example}{Example}%
\newtheorem{remark}{Remark}%

\theoremstyle{thmstylethree}%

\begin{document}

\title[$k$-Power Graphs of Finite Groups]{$k$-Power Graphs of Finite Groups}

\author*[1]{\fnm{Swathi} \sur{V V}}\email{swathivv14@gmail.com }

\author[1]{\fnm{M S} \sur{Sunitha}}\email{sunitha@nitc.ac.in}
\affil*[1]{\orgdiv{Department of Mathematics}, \orgname{National Institute of Technology Calicut}, \orgaddress{ \city{Calicut}, \postcode{673601}, \state{Kerala}, \country{India}}}

\abstract{For a finite group $G$ and for a fixed positive integer $k$, $k\geq 2$, the $k$-power graph of $G$ is an undirected simple graph with vertex set $G$ in which two distinct vertices $x$ and $y$ are adjacent if and only if $x^k=y$ or $y^k=x$. In this paper, we investigate some graph parameters such as number of edges, clique number, connectedness, etc. of $k$-power graphs of finite groups. Also find some properties of $k$-power graphs of finite cyclic groups, and finally we present an application.}

\keywords{$k$-power graphs, Power graphs, Graphs of groups.}

\maketitle

\section{Introduction}\label{sec1}
The directed power graph $\overrightarrow P(S)$ of a semigroup $S$ was defined by Kelarev and Quinn \cite{kelarev2000combinatorial}, as a digraph with vertex set $S$ in which there is an arc from a vertex $x$ to another vertex $y$ if and only if $y=x^n$ for some positive integer $n$.
Following this, the (undirected) power graph $P(S)$ of a semigroup $S$ was defined by Chakrabarty et al.\cite{chakrabarty2009undirected} as an undirected simple graph with vertex set $S$ in which two distinct vertices $x$ and $y$ are adjacent if and only if $x^n=y$ or $y^n=x$ for some positive integer $n$. The authors proved that power graph of a finite group $G$ is complete if and only if $G$ is cyclic of order $1$ or $p^m$ where $p$ is a prime and $m$ is a positive integer.
\medskip

We find a number of research publications where researchers describe various graph parameters for power graphs completely or in part. \cite{abawajy2013power,cameron2011power, cameron2019power,cameron2020connectivity,cameron2022matching,chattopadhyay2019vertex,chelvam2018power,doostabadi2015connectivity,ma2015chromatic} are some of the literatures which discuss numerous properties of power graphs of finite groups.

For fixed positive integer $k$ with $k\geq 2$ and for a semigroup $S$, Sriparna Chattopadhyay and Pratima Panigrahi \cite{chattopadhyay2017some} defined the $k-$power graph $P(S,k)$ as a graph with vertex set $S$ in which two distinct vertices $x$ and $y$ are adjacent if and only if $x^k=y$ or $y^k=x$. They studied cycle structures, connectedness
and symmetry of $k-$power graphs of cyclic groups. Clearly, for every positive integer $k\geq2$, $P(S,k)$ is a spanning subgraph of $P(S)$. 

In \cite{swathi2021square} the authors defined square graph of a finite group which is a particular $k$-power graph for $k=2$, and they investigated upper bounds for clique number and chromatic number. 

In this paper, notation and terminology of \cite{herstein2006topics} for graphs, \cite{chartrand2006introduction} for groups and \cite{nathanson2008elementary} for number theory are followed.

We denote the order of a group $G$ by $o(G)$ and order of an element $a$ in $G$ by $o(a)$. The identity element in $G$ is denoted by $e$. Also, a cyclic group of order $n$ is denoted by $\mathbb{Z}_n$.

The order of a graph $H$ is the number of vertices in that graph. We denote $deg(x)$ as the degree of the vertex $x$, which is the number of edges incident to $x$. The greatest distance between any two vertices of a connected graph $H$ is called the diameter of $H$ and is denoted by $diam(H)$. The complete graph of order $n$ is denoted by $K_n$ and a path of order $n$ by $P_n$. We denote $x\sim y$ if the vertices $x$ and $y$ are adjacent in the graph.

The greatest common divisor of two integers $m$ and $n$ are denoted as $gcd(m,n)$. For a positive integer $n>1$ and for an integer $a$ such that $gcd(a,n)=1$, the order of $a$ modulo $n$ is defined as the least positive integer $k$ such that $a^k\equiv 1(mod\, n)$, which is denoted by $ord_n(a)$. If $a$ has order $\phi(n)$, where $\phi(m)$ is the Euler totient function of $m$, then $a$ is called the primitive root modulo $n$.

In this paper we find some structural properties of $k$-power graph of a general finite group $G$ of order $n$ and $2\leq k \leq n$. Also we prove some properties of $P(\mathbb{Z}_n,k)$ and give an application also.
\section{$k$-power graphs}

Let $G$ be a finite group such that $o(G)=n$ and $P(G,k)$ be the $k$-power graph of $G$. For any two integers $k_1,k_2\geq 2,$ if $k_1\equiv k_2(mod \, n)$ then clearly $P(G,k_1)=P(G,k_2)$. We begin this section by noting that for a finite group $G$ of order $n$, the edges in $P(G,k)$ is at most $n-1$, since only the $k-$th power of each element is considered. We calculate the number of edges in $P(G,k)$ in the following theorem.
\begin{theorem}
Let $G$ be a finite group group of order $n$. Then the number of edges in the $k-$power graph of $G$ is given by 
$\mid E(P(G,k))\mid=n-\sum_{d\mid k_1} t_d-\sum_{d\mid k_2,d\nmid k_1}\frac{t_d}{2}$, where $k_1=gcd(k-1,n), k_2=gcd(k^2-1,n)$ and $t_d=$number of elements of order $d$ in $G$.
\end{theorem}
\begin{proof}
Corresponding to each element $a \in G$ there exists an edge in $P(G,k)$ between $a$ and $a^k$ if $a \neq a^k$. Now, if $x$ is an element $G$ which has order $d$ and $d \mid k-1$, then $x^k=x$. Hence, if $d \mid gcd(k-1,n)$, there are no edges corresponding to the elements of order $d$ in $G$. Now, edges corresponding to two elements $x$ and $y$ coincide if $x^k=y$ and $y^k=x$. In this case, $o(x) \mid {k^2-1}$ and $ o(y) \mid {k^2-1}$. Also if $x \in G$ has order $d$, $d \mid gcd(k^2-1,n)$ and $d \nmid gcd(k-1,n)$, then $x\neq x^k$ and edges corresponding to $x$ and $x^k$ coincide. Note that number of elements in $G$ of order $d$ is even if $d>2$, and if $d=2$ and $d \mid k_2$, then $d \mid k_1$.
\end{proof}
Since there are $\phi(d)$ elements of order $d$ in a cyclic group, the following corollary is obvious.
\begin{corollary}
Let $\mathbb{Z}_n$ be the cyclic group of order $n$. Then $\mid E(P(\mathbb{Z} _n,k))\mid=n-\sum_{d\mid k_1} \phi(d)-\sum_{d\mid k_2,d\nmid k_1}\frac{\phi(d)}{2}$.
\end{corollary}

\begin{example}
The $k$-power graphs of the symmetric group on 3 elements ($S_3$) for $k=2,3,4$ and $5$ are given in figures 1,2,3 and 4 respectively, where $\sigma_0=e,\sigma_1$=(1 2 3), $\sigma_2=$(1 3 2),   $\tau_1=$(1 2), $\tau_2$=(2 3) and $\tau_3$=(1 3) .\\

\end{example}

\unitlength 1mm 
\linethickness{0.4pt}
\ifx\plotpoint\undefined\newsavebox{\plotpoint}\fi 
\unitlength 1mm 
\linethickness{0.4pt}
\ifx\plotpoint\undefined\newsavebox{\plotpoint}\fi 
\begin{picture}(68.303,52.518)(0,0)
\multiput(77.053,53.268)(.03125,.03125){8}{\line(0,1){.03125}}
\unitlength 1mm 
\linethickness{0.4pt}
\ifx\plotpoint\undefined\newsavebox{\plotpoint}\fi 
\unitlength 1mm 
\linethickness{0.4pt}
\ifx\plotpoint\undefined\newsavebox{\plotpoint}\fi 
\begin{picture}(60.303,53.518)(0,0)
\multiput(77.053,53.268)(.03125,.03125){8}{\line(0,1){.03125}}
\unitlength 1mm 
\linethickness{0.4pt}
\ifx\plotpoint\undefined\newsavebox{\plotpoint}\fi 
\unitlength 1mm 
\linethickness{0.4pt}
\ifx\plotpoint\undefined\newsavebox{\plotpoint}\fi 
\begin{picture}(60.303,53.518)(0,0)
\multiput(77.053,53.268)(.03125,.03125){8}{\line(0,1){.03125}}
\multiput(20.053,48.268)(-.03358209,-.059701493){201}{\line(0,-1){.059701493}}
\put(20.303,48.518){\line(0,-1){22.75}}
\put(20.553,48.268){\line(3,-5){12.75}}
\put(20.053,48.268){\circle*{2.9}}
\put(20.053,51.268){\makebox(0,0)[cc]{$\sigma_0$}}
\put(13.553,37.018){\circle*{2.9}}
\put(9.550,37.018){\makebox(0,0)[cc]{$\tau_1$}}
\put(20.553,26.018){\circle*{2.9}}
\put(20.553,22.018){\makebox(0,0)[cc]{$\tau_2$}}
\put(33.553,26.018){\circle*{2.9}}
\put(33.553,22.018){\makebox(0,0)[cc]{$\tau_3$}}
\put(32.053,48.268){\circle*{2.9}}
\put(32.553,51.268){\makebox(0,0)[cc]{$\sigma_1$}}
\put(38.553,37.768){\circle*{2.9}}
\put(42.553,37.768){\makebox(0,0)[cc]{$\sigma_2$}}
\put(25.553,10.768){\makebox(0,0)[cc]{Figure 1: $P(S_3,2)$}}
\multiput(32.053,48.268)(.03358209,-.053482587){201}{\line(0,-1){.053482587}}
\end{picture}

\put(20.303,48.053){\line(2,0){12.65}}
\put(20.053,47.5){\line(2,-1){20}}
\put(20.053,48.268){\circle*{2.9}}
\put(20.053,51.268){\makebox(0,0)[cc]{$\sigma_0$}}
\put(13.553,37.018){\circle*{2.9}}
\put(9.550,37.018){\makebox(0,0)[cc]{$\tau_1$}}
\put(20.553,26.018){\circle*{2.9}}
\put(20.553,22.018){\makebox(0,0)[cc]{$\tau_2$}}
\put(33.553,26.018){\circle*{2.9}}
\put(33.553,22.018){\makebox(0,0)[cc]{$\tau_3$}}
\put(32.053,48.268){\circle*{2.9}}
\put(32.553,51.268){\makebox(0,0)[cc]{$\sigma_1$}}
\put(38.553,37.768){\circle*{2.9}}
\put(42.553,37.768){\makebox(0,0)[cc]{$\sigma_2$}}


\end{picture}

\put(20.053,48.268){\circle*{2.9}}
\put(20.053,51.268){\makebox(0,0)[cc]{$\sigma_0$}}
\put(13.553,37.018){\circle*{2.9}}
\put(9.550,37.018){\makebox(0,0)[cc]{$\tau_1$}}
\put(20.553,26.018){\circle*{2.9}}
\put(20.553,22.018){\makebox(0,0)[cc]{$\tau_2$}}
\put(33.553,26.018){\circle*{2.9}}
\put(33.553,22.018){\makebox(0,0)[cc]{$\tau_3$}}
\put(32.053,48.268){\circle*{2.9}}
\put(32.553,51.268){\makebox(0,0)[cc]{$\sigma_1$}}
\put(38.553,37.768){\circle*{2.9}}
\put(42.553,37.768){\makebox(0,0)[cc]{$\sigma_2$}}
\put(25.553,10.768){\makebox(0,0)[cc]{Figure 2: $P(S_3,3)$}}

\end{picture}

\unitlength 1mm 
\linethickness{0.4pt}
\ifx\plotpoint\undefined\newsavebox{\plotpoint}\fi 
\unitlength 1mm 
\linethickness{0.4pt}
\ifx\plotpoint\undefined\newsavebox{\plotpoint}\fi 
\begin{picture}(60.303,53.518)(0,0)
\multiput(77.053,53.268)(.03125,.03125){8}{\line(0,1){.03125}}
\unitlength 1mm 
\linethickness{0.4pt}
\ifx\plotpoint\undefined\newsavebox{\plotpoint}\fi 
\unitlength 1mm 
\linethickness{0.4pt}
\ifx\plotpoint\undefined\newsavebox{\plotpoint}\fi 
\begin{picture}(60.303,53.518)(0,0)
\multiput(77.053,53.268)(.03125,.03125){8}{\line(0,1){.03125}}
\unitlength 1mm 
\linethickness{0.4pt}
\ifx\plotpoint\undefined\newsavebox{\plotpoint}\fi 
\unitlength 1mm 
\linethickness{0.4pt}
\ifx\plotpoint\undefined\newsavebox{\plotpoint}\fi 
\begin{picture}(77.303,53.518)(0,0)
\multiput(77.053,53.268)(.03125,.03125){8}{\line(0,1){.03125}}
\multiput(20.053,48.268)(-.03358209,-.059701493){201}{\line(0,-1){.059701493}}
\put(20.303,48.518){\line(0,-1){22.75}}
\put(20.553,48.268){\line(3,-5){12.75}}
\put(20.053,48.268){\circle*{2.9}}
\put(20.053,51.268){\makebox(0,0)[cc]{$\sigma_0$}}
\put(13.553,37.018){\circle*{2.9}}
\put(9.550,37.018){\makebox(0,0)[cc]{$\tau_1$}}
\put(20.553,26.018){\circle*{2.9}}
\put(20.553,22.018){\makebox(0,0)[cc]{$\tau_2$}}
\put(33.553,26.018){\circle*{2.9}}
\put(33.553,22.018){\makebox(0,0)[cc]{$\tau_3$}}
\put(32.053,48.268){\circle*{2.9}}
\put(32.553,51.268){\makebox(0,0)[cc]{$\sigma_1$}}
\put(38.553,37.768){\circle*{2.9}}
\put(42.553,37.768){\makebox(0,0)[cc]{$\sigma_2$}}
\put(25.553,10.768){\makebox(0,0)[cc]{Figure 3: $P(S_3,4)$}}

\end{picture}


\end{picture}

\put(20.053,48.268){\circle*{2.9}}
\put(20.053,51.268){\makebox(0,0)[cc]{$\sigma_0$}}
\put(13.553,37.018){\circle*{2.9}}
\put(9.550,37.018){\makebox(0,0)[cc]{$\tau_1$}}
\put(20.553,26.018){\circle*{2.9}}
\put(20.553,22.018){\makebox(0,0)[cc]{$\tau_2$}}
\put(33.553,26.018){\circle*{2.9}}
\put(33.553,22.018){\makebox(0,0)[cc]{$\tau_3$}}
\put(32.053,48.268){\circle*{2.9}}
\put(32.553,51.268){\makebox(0,0)[cc]{$\sigma_1$}}
\put(38.553,37.768){\circle*{2.9}}
\put(42.553,37.768){\makebox(0,0)[cc]{$\sigma_2$}}
\put(25.553,10.768){\makebox(0,0)[cc]{Figure 4: $P(S_3,5)$}}
\multiput(32.053,48.268)(.03358209,-.053482587){200}{\line(0,-1){.053482587}}
\end{picture}

\begin{example}
The $k$-power graph of $\mathbb{Z}_4$ when $k=2$ is the following tree.

\end{example}
\unitlength 1mm 
\linethickness{0.4pt}
\ifx\plotpoint\undefined\newsavebox{\plotpoint}\fi 
\begin{picture}(81.838,30.525)(0,0)
\put(27.82,20.871){\line(0,-1){8.75}}

\multiput(10.57,11.371)(8.625,.03125){4}{\line(1,0){8.625}}
\put(45.57,11.121){\circle*{2.9}}
\put(45.57,8.018){\makebox(0,0)[cc]{1}}
\put(27.82,11.371){\circle*{2.9}}
\put(27.82,8.018){\makebox(0,0)[cc]{2}}
\put(27.82,21.871){\circle*{2.9}}
\put(27.82,25.018){\makebox(0,0)[cc]{0}}
\put(10.82,11.371){\circle*{2.9}}
\put(10.82,8.018){\makebox(0,0)[cc]{3}}
\put(42.553,0){\makebox(0,0)[cc]{Figure 5: $P(\mathbb{Z}_4,2$)}}
\end{picture}\\

It is clear from the Figures 1-5 that $k$-power graphs of a finite group need not be connected always. The following proposition gives a characterization to the cyclic groups such that its $k-$power graph is connected.

\begin{proposition} \cite{chattopadhyay2017some}
The graph $P(\mathbb{Z}_n,k)$ is connected if and only if $n\mid k^m$ for some $m\in \mathbb{N}.$ Also in this case $P(\mathbb{Z}_n,k)$ is a tree.
\end{proposition}

For any finite group $G$, the characterization for $P(G,k)$ to be connected is given in the following theorem.
\begin{theorem}
Let $G$ be a finite group. $P(G,k)$ is connected if and only if for every $x \in G$, $o(x)\mid k^n$ for some $n \in \mathbb{N}$. For $x \in G$, let $n_x= min\{n \in \mathbb{N}: o(x)\mid k^n \}$, then $diam(P(G,k)) \leq 2 max \{ n_x: x \in G \}$.
\end{theorem}
\begin{proof}
Let $x\in G$ and $o(x)\mid k^n$ for some $n\in \mathbb{N}$. Suppose $n_x=min\{n\in \mathbb{N}: o(x)\mid k^n\}$. Then $x$ is connected to $e$ through a path of length $n_x$ in $P(G,k)$. Hence $P(G,k)$ is  connected if for all $x\in G$, $o(x) \mid k^n$ for some $n\in \mathbb{N}$, and $diam(P(G,k)) \leq 2 max \{ n_x: x \in G \}$. The converse also follows.
\end{proof}
The following theorem proves that the clique number is bounded above.
\begin{theorem}
The clique number of $P(G,k)$, $\omega(P(G,k)) \leq 3$ for any finite group $G$ and $\omega(P(G,k))=3$ if and only if there exist elements in $G$ of order $m$ such that $m>3, m \mid k^3-1$ and $m\nmid k-1.$
\end{theorem}
\begin{proof}
If possible, suppose there exist elements $x,y,z, w \in G$ which induces $K_4$. Then the following six conditions must be satisfied.
\begin{enumerate}
    \item $x^k=y$ or $y^k=x$
    \item $y^k=z$ or $z^k=y$
    \item $z^k=w$ or $w^k=z$
    \item $w^k=x$ or $x^k=w$
    \item $x^k=z$ or $z^k=x$
    \item $y^k=w$ or $w^k=y$ 
    \end{enumerate}
Suppose $x^k=y$ in condition 1. Then conditions 4 and 5 implies that $w^k=x$ and $z^k=x$, which is not possible by condition 3. Hence, $P(G,k)$ has no subgraph isomorphic to $K_4$. Hence, the clique number is atmost 3.

 Now suppose $\omega(P(G,k))=3$. Then there exists an element $a \neq e$ such that $a^k \neq a,e$, $ a^{k^2} \neq a,e$ and $a^{k^3}=a$. Then $o(a) \mid k^3-1$ and $o(a) \nmid k-1$. Note that, if $m \mid k^3-1$ and $m \mid k^2-1$ then $m \mid k-1$. The converse also follows.
\end{proof}

Next theorem states that the chromatic number also bounded above by 3.

\begin{theorem}
Let $G$ be a finite group, then the chromatic number $\chi(P(G,k)) \leq 3$.
\end{theorem}
\begin{proof}
Choose a vertex $x$ from a connected component of $P(G,k)$ and colour 1 is assigned. Assign colour 2 to the vertex $x^k$. Label $y$ to all vertices adjacent to $x$ except $x^k$. Then clearly $y^k=x$. If $(x^k)^k=y$, colour 3 is assigned to $y$, otherwise assign colour 2. Now all the vertices adjacent to $y$ are labelled as $z$. Then $z^k=y$ and colour $z$ using colour 3. Now label $w$ to all the vertices adjacent to $z$, then $w^k=z$. Again colour 3 is given to $w,$ if $(x^k)^k=w$, otherwise give colour 2. Proceed with the same process until every vertices in that component get a colour. Apply the same procedure to each connected components of $P(G,k)$. So we conclude that the minimum number of colours required to properly colour $P(G,k)$ is 3.
\end{proof}

 The question of whether the $k$-power graph is perfect therefore arises naturally. The following remark gives an answer for that.

\begin{remark}
 The $k$-power graph of a finite group need not be perfect. For example, $P(\mathbb{Z}_{31},2)$ is a union of an isolated vertex and six $5-$cycles, which has chromatic number $3$ and clique number equal to $2$.
\end{remark}

The groups given by the presentation $Q_{4n}=\langle a,b : a^n=b^2, a^{2n}=1, b^{-1}ab=a^{-1} \rangle$ are the generalized quaternion groups. 

\begin{theorem}
Let $G$ be a finite group, then $P(G,k)$ is a star graph if and only if one of the following holds.
\begin{itemize}
    \item $o(x)\mid k \, \,  \forall x \in G$ 
    \item $G=\mathbb{Z}_4$ and $k=2$ 
    \item $G=Q_8$ and $k=2 \, \text{or}\,\, 6$ 
\end{itemize}
\end{theorem}
\begin{proof}
Let $P(G,k)$ is a star graph and let there exist $x\in G$ such that $o(x)\nmid k$, then $x^k\neq e$. Suppose $x^k=y$ for some $y\in G$. If $y=x$, then $z^k=x$ for all $z\in G, z\neq x$ which is not possible since $e^k\neq x$.
If $y \neq x$, then $z^k= y$ for all $z \in G, z \neq y$ and $y^k=e$. 

If $x^2=y$, then $y^2=e$, which implies $G$ is a group with unique involution and all other elements have order $4$. Hence $G=\mathbb{Z}_4$ with $k=2$ or $G=Q_8$ with $k=2 \, \text{or} \, 6$ [A $p$-group with unique subgroup of order $p$ is either cyclic or generalized quaternion].\\
If $x^2\neq y$, then $x^k=y$ and $(x^2)^k=y$ which implies $x^k=e$ and hence $o(x)\mid k$.

For the converse, if $o(x)\mid k, \forall x\in G$, then $x^k=e, \forall x\in G$, and hence $P(G,k)$ is a star graph. The $k$-power graphs of $\mathbb{Z}_4$ and $Q_8$ also star graphs when $k=2 \, \text{and}\, 6$.
\end{proof}
Clearly, if $k-1 \mid o(x)$ for all $x \in G$, then $x^k=x$ and every vertex in $G$ is an isolated vertex in $P(G,k)$, and vice versa. Hence the following theorem is immediate.
\begin{theorem}
$P(G,k)$ is an empty graph if and only if $o(x) \mid k-1$ for every $x \in G$.
\end{theorem}
We have already seen that $P(G,k)$ need not be connected, and if $P(G,k)$ is connected, then it is a tree. The following theorem characterizes groups whose $k-$power graphs are forests.
\begin{theorem}
Let $G$ be a finite group. $P(G,k)$ is a forest if and only if there exist no element in $G$ of order $m>1$ such that $gcd(k,m)=1$ and $ord_m(k)>2$.
\end{theorem}
\begin{proof}
By Euler's theorem in number theory, if $m$ and $k$ are co-prime integers, then $a^{\phi(m)}\equiv 1 (mod \, m)$. Hence, if there exists an element $x$ in $G$ of order $m $ such that $gcd(m,k)=1$ and $ord_{m}(k)=l>2$, then $x^{k^l}=x$, and $x,x^k, x^{k^2},...x^{k^l}$ is a cycle in $P(G,k)$ of length $l$ and hence $P(G,k)$ is not a forest.

Conversely, if $P(G,k)$ is not a forest, then there is a cycle $x,x^k,x^{k^2},.. x^{k^l}=x$ in $P(G,k)$, which implies $o(x) \mid k^l-1$. Hence $k^{l}\equiv 1 (mod \, o(x))$ which holds only if $gcd(k,o(x))=1$.
\end{proof}

\section{Cyclic groups}
Let $gcd(n,k)=d$, then the congruent relation $kx\equiv a(mod\, n)$ has at most $d$ solutions. Hence, if $G$ is a cyclic group of order $n$, then the maximum degree of its $k-$power graph is at most $d+1$. Then following theorem describes degree of each element in a cyclic group.
\begin{theorem}
Consider the cyclic group $\mathbb{Z}_n=\{0,1,2,...n-1\}$. Let $d=(n,k)$ and $a \in \mathbb{Z}_n$. Then,

If $d\nmid a$,  $deg(a)=\begin{cases}0& \text{if} \,\,   o(a)\mid k-1\\ \,  1& \text{if} \, \,   o(a)\nmid k-1
\end{cases}$

If $d \mid a$, $deg(a)=\begin{cases}d-1& \text{if}\, \, ka\equiv a(mod\, n)\, \, \text{and}\, \, \, o(a)\mid k-1\\
d& \text{if}\, \, ka\not\equiv a(mod\, n)\, \, and\, \, \, o(a)\mid k-1\\
d&\text{if}\, \, ka\equiv a(mod\, n)\, \, \text{and}\, \, \, o(a)\nmid k-1\\
d+1&\text{if}\, \, ka\not\equiv a(mod\, n)\, \, \text{and}\, \, \, o(a)\nmid k-1
\end{cases}$
\end{theorem}
\begin{proof}
If $d\nmid a$, the congruence relation $kx\equiv a(mod\, n)$ has no solution. Therefore, the vertex $a$ is adjacent to $a^k$ only. In this case, if $o(a)\mid k-1$, then $a^k=a$ and $deg(a)=0$ and if $o(a) \nmid k-1$, then $a^k\neq a$ and $deg(a)=1$.

Next suppose $d\mid a$. Then the congruence relation $kx\equiv a(mod\, n)$ has exactly $d$ solutions. Now we have two cases,\\
\textbf{case 1:} $ka\equiv a(mod\, n)$\\
In this case, $a$ itself is a solution of the congruence relation. Hence $deg(a)=d-1$ if $o(a)\mid k-1$ and $deg(a)=d$ if $o(a)\nmid k-1$.\\
\textbf{case 2:} $ka\not\equiv a(mod\, n)$\\
Here, all the $d$ solutions of the congruence relation are different from $a$, Hence $deg(a)=d$ if $o(a)\mid k-1$ and $deg(a)=d+1$ if $o(a)\nmid k-1$.
\end{proof}

For a positive integer $m$, let $\pi(m)=\{p: p\mid m, p\, \text{is a prime}\}$. The following theorem gives another characterization for $P(\mathbb{Z}_n,k)$ to be connected.

\begin{theorem}
$P(\mathbb{Z}_n,k)$ is connected if and only if $\pi(n) \setminus \pi(k)=\phi$.
\end{theorem}
\begin{proof}
Suppose $\pi(n)\setminus\pi(k)\neq \phi$, and let $p\in \pi(n)$ such that $p\notin \pi(k)$. Then $gcd(n,k)=1$ and $k^{p-1}\equiv 1(mod\, p).$ Let $m=exp_p(k)$ and let $x$ be an element in $G$ such that $o(x)=p$. Then $x^{k^m}=x$ and $x$ is not connected to $e$ through any path in $P(\mathbb{Z}_n,k)$. Hence $P(\mathbb{Z}_n,k)$ is not connected.

Conversely suppose $\pi(n)\setminus\pi(k)=\phi$. Let $n=p_1^{\alpha_1}p_2^{\alpha_2}...p_s^{\alpha_s}$ and $k=p_1^{\beta_1}p_2^{\beta_2}...p_r^{\beta_r}$ be the prime factorization of $n$ and $k$ respectively, where $r\leq s$. Let $m=lcm(\beta_1,\beta_2,...\beta_r)$. Then $n\mid k^m$, and hence by Proposition 3, $P(\mathbb{Z}_n,k)$ is connected.
\end{proof}

 The identity element $e$ is an isolated vertex in $P(\mathbb{Z}_n,k)$ if $gcd(n,k)=1$ since, if $x^k=e$ for some $x\in \mathbb{Z}_n$ then $o(x)\mid k$ and hence $o(x)\mid gcd(n,k)$.
\begin{proposition}
Let $gcd(n,k)=1$. If $x\sim y$ in $P(\mathbb{Z}_n,k)$, then $o(x)=o(y)$.
\end{proposition}
\begin{proof}
Let $x\sim y$, then $kx\equiv y (mod\, n) \implies o(x)kx\equiv o(x)y(mod \, n) \implies 0\equiv o(x)y(mod\, n) \implies o(y)\mid o(x)$. Also $o(y)kx\equiv o(y)y(mod\, n) \implies o(y)kx\equiv 0(mod\, n) \implies o(x)\mid o(y)k\implies o(x)\mid o(y)$, since $gcd(o(x),k)=1$.
\end{proof}

The following theorem states that the number of connected components is bounded above.
\begin{theorem}
If $gcd(n,k)=1$, then the number of connected components of $P(\mathbb{Z}_n,k)$, $c(P(\mathbb{Z}_n,k))\geq \tau(n)$, where $\tau(n)$ is the number of divisors of $n$, and $c(P(\mathbb{Z}_n,k))=\tau(n)$ if and only if $k$ is a primitive root modulo $d$, for every divisor $d$ of $n$.
\end{theorem}
\begin{proof}
By Proposition $3.4$, all the vertices in a component have same order. Hence, the inequality follows since an element of order $d$ exists in $\mathbb{Z}_n$ if $d\mid n$.

Suppose $k$ is a primitive root modulo $d$, for every $d\mid n$. Then $ord_d(k)=\phi(d)$. Hence, if $x$ is an element in $G$ of order $d$, then $\phi(d)$ is the smallest positive integer such that $x^{k^{\phi(d)}}\equiv x(mod\, d)$. Hence $x,x^k,x^{k^2},...,x^{k^{\phi(d)-1}}$ are the distinct $\phi(d)$ elements of order $d$ in $G$ and are in the same component of $P(\mathbb{Z}_n,k)$.

Conversely, suppose $k$ is not a primitive root modulo $d$ for some $d \mid n$. Then, $ord_d(k)=m<\phi(d)$ and hence $k^m\equiv 1(mod\, d)$ and $x^{k^m}=x$ for every element $x$ of order $d$. Hence, the component containing $x$ has exactly $m$ vertices. Therefore, the elements in $\mathbb{Z}_n$ of order $d$ contribute at least two components.
\end{proof}
\begin{proposition}\cite{chattopadhyay2017some}
If $gcd(n,k)=1$ then any component of  $P(\mathbb{Z}_n,k)$ is an isolated vertex, the complete graph $K_2$ or a cycle of length at least 3.
\end{proposition}

The following theorem shows that the converse of the Proposition 3 is also satisfies.
\begin{theorem}
Any component of $P(\mathbb{Z}_n,k)$
 is an isolated vertex, the complete graph $K_2$ or a cycle of length at least $3$ if and only if $gcd(n,k)=1$.
\end{theorem}
\begin{proof}
Suppose $gcd(n,k)=d>1$. Then, the congruence $kx\equiv 1(mod\, n)$ has no solutions. Considering $1$ and $k$ as elements of $\mathbb{Z}_n$, $1 \sim k$ in $P(\mathbb{Z}_n,k)$. Hence, $deg(1)=1$. Also, $kx\equiv k(mod\, n)$ has $d$ solutions and so $deg(k)\geq2$. Therefore, the edge between $1$ and $k$ in $P(\mathbb{Z}_n,k)$ is not a part of a cycle nor $K_2$ as a component. Hence, the necessary part follows. The sufficient part follows from Proposition $14$.
\end{proof}

\begin{theorem}
Let $n$ be even and $\frac{n}{2}$ odd, then $P(\mathbb{Z}_n,\frac{n}{2})=2K_{1,\frac{n}{2}-1}$ and $P(\mathbb{Z}_n,\frac{n}{2}+1)=\frac{n}{2}P_2$.
\end{theorem}
\begin{proof}
Let $n$ be even and $\frac{n}{2}$ odd. If $k=\frac{n}{2}$, then using Theorem $10$,
$deg(a)=\frac{n}{2}-1$ if $a=0,\frac{n}{2}$ and $deg(a)=1$ if $a\neq 0,\frac{n}{2}.$ Also, $\frac{n}{2}a\equiv 0 (mod\, n)$ if $a$ is even and hence all the even elements are adjacent to the vertex 0, and $\frac{n}{2}a\equiv \frac{n}{2}(mod\, n)$ if $a$ is odd and hence all the odd elements are adjacent to the vertex $\frac{n}{2}$.

If $k=\frac{n}{2}+1$, for $a\in \mathbb{Z}_n$, $(\frac{n}{2}+1)a\equiv a(mod\, n)$ if $a$ is even, and $(\frac{n}{2}+1)a\equiv \frac{n}{2}+a(mod\, n)$ if $a$ is odd. Hence, the edges in $P(\mathbb{Z}_n,k)$ are between the vertices $a$ and $\frac{n}{2}+a$ only.
\end{proof}
\section{Application}
Consider the directed $k$-power graph of a group $G$, $\overrightarrow P(G,k)$, defined as a digraph with vertex set $G$ in which there is an arc from a vertex $a$ to another vertex $b$ if and only if $b=a^k$. In this section, we present a riddle and solve that riddle using directed $k$-power graphs of cyclic groups.\\\\
\textbf{Shifting Chair Problem}: \textit{In a game, $n$ chairs are evenly spaced in a circle and $n$ people are assigned unique numbers from $1$ to $n$. People should sit in ascending order in the clockwise direction when the first whistle blows. For each of the whistles thereafter, the individual who is assigned with number $i$ should acquire his next chair by skipping exactly $i-1$ number of chairs in the clockwise direction. The problem is to determine the smallest number of whistles required to ensure that each chair is occupied by exactly one person.}\\

Consider the above problem, representing the movement of the individuals from their initial seats to the position after $k^{\text{th}}$ whistle by a directed edge, we get the directed $k$-power graph of $\mathbb{Z}_n$. Then the problem is to find the smallest $k>1$ such that, $od(a)=id(a)=1$ in $\overrightarrow P(\mathbb{Z}_n,k)$ for all $a\in \mathbb{Z}_n$, where $od(a)$ is the outdegree of $a$ and $id(a)$ is the indegree of $a$, and the following theorem present the solution.

\begin{theorem}
In $\overrightarrow P(\mathbb{Z}_n,k)$, $id(a)=od(a)=1$ for all $a\in \mathbb{Z}_n$ if and only if $gcd(n,k)=1$.
\end{theorem}
\begin{proof}
If $gcd(n,k)=1$, then for each $a\in \mathbb{Z}_n$, the congruence $kx\equiv a(mod\, n)$ has exactly one solution, so $id(a)=1$. Also $od(a)=1$ for all $a\in \mathbb{Z}_n$ since $a$ is adjacent to $a^k$.

Conversely suppose $gcd(n,k)=d>1$. Then $kx\equiv 1(mod\,n)$ has no solutions, so $id(1)=0$.
\end{proof}
\section{Declarations}

\textbf{Conflict of interest} On behalf of all authors, the corresponding author states that there is no conflict of interest.
\section*{Acknowledgements}
The first author gratefully acknowledges the financial support of Council of Scientific and Industrial Research, India (CSIR) (Grant No-09/874(0029)/2018-EMR-I).
The authors would like to thank the DST, Government of India, for providing support to carry out the work under the scheme `FIST' (No.SR/FST /MS-I/2019/40).




\bibliography{sn-reference}


\end{document}